\documentclass[12pt,a4paper,reqno]{amsart}
\usepackage{geometry}
\usepackage{graphicx}
\usepackage{placeins}
\usepackage[numbers]{natbib}
\DeclareMathOperator*{\Res}{Res}
\newtheorem{theorem}{Theorem}
\newtheorem{lemma}{Lemma}
\theoremstyle{remark}
\newtheorem*{remark}{Remark}

\numberwithin{equation}{section}

\begin{document}
\title[Mean-value of the $n$\textsuperscript{th} derivative of zeta]{A discrete mean-value theorem for the higher derivatives of the Riemann zeta function}

\author{Christopher Hughes}
\address{Department of Mathematics, University of York, York, YO10 5DD, United Kingdom}
\email{christopher.hughes@york.ac.uk}
\author{Andrew Pearce-Crump}
\address{Department of Mathematics, University of York, York, YO10 5DD, United Kingdom}
\email{aepc501@york.ac.uk}

\date{2nd March 2022}

\begin{abstract}
We show that the $n$\textsuperscript{th} derivative of the Riemann zeta function, when summed over the non-trivial zeros of zeta, is real and positive/negative in the mean for $n$ odd/even, respectively. We show this by giving a full asymptotic expansion of these sums.
\end{abstract}

\maketitle

\section{Shanks' Conjecture and statement of results}

Shanks' Conjecture from Shanks~\cite{ShanksConj} says for $\rho = \beta + i \gamma$ a zero of $\zeta(s)$, that
\[
\zeta '(\rho) \text{ is real and positive in the mean. }
\]

This conjecture was first proven by evaluating the discrete first moment of $\zeta'(\rho)$ to leading order in Conrey, Ghosh, Gonek~\cite{CGG88} and for all other terms in Fujii~\cite{FujiiShanks}, with a minor correction to its lowest order terms in Fujii~\cite{FujiiDist}.  We state the precise asymptotic formula shortly, in Theorem~\ref{Shanks' Conjecture}. Other proofs of Shanks' Conjecture have since been given (either explicitly or implied), for example, in Trudgian in \cite{Trud} and in Stopple~\cite{Stopp}.

In this paper we will generalise this to all higher derivatives, by giving explicit asymptotic formulae (including the lower order terms) for the sums of the $n$\textsuperscript{th} derivatives of $\zeta (s)$ evaluated at the zeros $\rho$ of $\zeta (s)$. From the results of Theorem~\ref{General SC} it will follow that
\begin{align*}
\zeta^{(n)}(\rho) \text{ is real and }
    \left\{ \begin{aligned}
        \text{positive } \\
        \text{negative}
    \end{aligned}\right\} \text{ in the mean if $n$ is }
     \left\{ \begin{aligned}
        \text{odd } \\
        \text{even}
    \end{aligned}\right\}.
\end{align*}

Throughout we assume $T$ is sufficiently large and that $|T-\gamma| \gg \frac{1}{\log T}$, where $\gamma$ is the imaginary part of any zero $\rho$. We write $s= \sigma + it$ where $\sigma, t \in \mathbb{R}$ for a general complex number $s \in \mathbb{C}$. We let $C_0$ and $C_1$ be the coefficients in the Laurent expansion of $\zeta (s)$ about $s=1$, which is given by
\[
\zeta(s) = \frac{1}{s-1} + C_0 + C_1 (s-1) + ...
\]

It is clear that Shanks' conjecture follows from the following result.
\begin{theorem}[Fujii] \label{Shanks' Conjecture}
With the notation and assumptions given above, we have
\begin{equation*}
\sum_{0 < \gamma \leq T} \zeta' (\rho) = \frac{T}{4 \pi} \log ^2 \left( \frac{T}{2 \pi} \right) + (-1 + C_0) \frac{T}{2 \pi} \log \left( \frac{T}{2 \pi} \right) + (1 - C_0 - C_0^2 + 3 C_1) \frac{T}{2 \pi} +E(T)
\end{equation*}
where, unconditionally,
\[
E(T) = O \left(T \mathrm{e}^{-C \sqrt{\log T}} \right)
\]
where $C$ is a positive constant. Conditionally upon the Riemann Hypothesis we have
\[
E(T) = O \left(T^{\frac{1}{2}} \log ^{\frac{7}{2}} T \right).
\]
\end{theorem}

The aim of this paper is to generalise the earlier result on Shanks' Conjecture to all $n$\textsuperscript{th} derivatives of zeta. Kaptan, Karabulut, Y{\i}ld{\i}r{\i}m  \cite{KKY} found the main term for the mean of $\zeta^{(n)}(\rho)$.
\begin{theorem}[Kaptan, Karabulut and Y{\i}ld{\i}r{\i}m] \label{KKY}
With the notation and assumptions given above, we have
\begin{equation*}
\sum_{0 < \gamma \leq T} \zeta^{(n)} (\rho) =
 (-1)^{n+1} \frac{1}{n+1} \frac{T}{2 \pi} \log^{n+1} \left( \frac{T}{2 \pi} \right) + O (T \log^{n} T).
\end{equation*}
\end{theorem}

Similar to us, and all other previous papers on this and other calculations, they start by rewriting the sum as a Cauchy residue integral, but they only proceed to calculate the main term in their subsequent calculations, while we also calculate the full asymptotic and error terms. (In fact, we only became aware of their paper while writing up our results).

To give the lower order terms in the expansion, in addition to the coefficients $C_j$ defined previously, let $A_j$ be the coefficients in the Laurent expansion for $\frac{\zeta' (s)}{\zeta (s)}$ about $s=1$, given by
\[
\frac{\zeta' (s)}{\zeta (s)} = - \frac{1}{s-1} + \sum_{j=0}^\infty A_j (s-1)^j.
\]
Note that the $A_j$ are related to the $C_j$ by the following recursive formula, as shown in Israilov~\cite{LaurentZetaPrime}, with equivalent forms in Maslanka~\cite{LiCoefficients} and Bomberi~\cite{Bombieri}:

\[
A_n =
\begin{cases}
C_0 &\text{ if $n=0$}\\
(n+1)C_n - \sum_{j=0}^{n-1} A_j C_{n-1-j} &\text{ if $n \geq 1$}.\\
      \end{cases}
\]

\begin{remark}
Our constants $C_j$ are related to the Stieltjes constants $\gamma_j$ which are normally used in the Laurent expansion for $\zeta (s)$ about $s=1$ by
\[
C_j = \frac{(-1)^j}{j!}\gamma_j.
\]
We have decided to use $C_j$ for simplicity in our formulae, and to remain consistent with Fujii's papers on Shanks' Conjecture. Clearly our coefficients $A_j$ are then also related to $\gamma_j$.
\end{remark}

\begin{theorem}\label{General SC}
With the notation and assumptions given above, we have
\begin{align*}
\sum_{0 < \gamma \leq T} \zeta^{(n)} (\rho) &=
 (-1)^{n+1} \frac{1}{n+1} \frac{T}{2 \pi} \log^{n+1} \left( \frac{T}{2 \pi} \right)   \\
&+ (-1)^{n+1} \sum_{k=0}^n \binom{n}{k}  (-1)^ {k}  k!  \left(-1 + \sum_{j=0}^k (-1)^j C_j \right) \frac{T}{2 \pi} \log^{n-k} \left( \frac{T}{2 \pi} \right) \\
&+ n! A_n  \frac{T}{2 \pi}  + E_n(T)
\end{align*}
where
\[
E_n(T) = O \left(T  \mathrm{e}^{-C \sqrt{\log T}} \right)
\]
with $C$ is a positive constant. If we assume the Riemann Hypothesis, then
\[
E_n(T) = O \left(T^{\frac{1}{2}} \log^{n+\frac{5}{2}} T \right).
\]
In the case $n=1$, we are able to show that under RH, 
\[
E_1(T) = O \left(T^{\frac{1}{2}} \log^{\frac{13}{4}} T \right).
\]
\end{theorem}

\begin{remark}
    After this paper was published, the second author in his PhD thesis \cite[Section 3.1.6]{theAPC} showed that a similar improvement can be made for any $n$, that is the error term $E_n(T)$ under RH in the theorem can be replaced with $O \left(T^{\frac{1}{2}} \log^{n+\frac{9}{4}} T \right)$.
\end{remark}

\begin{remark}
    After the publication of this paper, the authors \cite{HPC} found a neater way of writing the result, noting it can be expressed as an integral
    \begin{multline*}
		\sum_{0 < \gamma \leq T} \zeta ^{(n)} (\rho) 
        =  \frac{n!}{2 \pi} \int_{1}^{T} \left( A_n + \frac{(-1)^{n+1} L^{n+1}}{(n+1)!} +(-1)^{n+1} \sum_{m=0}^{n} \frac{ C_{n-m} L^m }{m!} \right) \ dt + E_n(T),
	\end{multline*}
    where $L = \log \frac{t}{2\pi}$.

\end{remark}

\section{Overview of the Proof}
In this section we give a brief outline of the paper, including the method we will follow to prove Theorem~\ref{General SC}.

In Section~\ref{sect:prelim}, we state some preliminary lemmas that we will use throughout the paper.

In Section~\ref{sect:beginProof}, we begin by considering the integral $S$ given by
\begin{equation*}
S= \frac{1}{2 \pi i} \int_R \frac{\zeta '}{\zeta}(s) \zeta ^{(n)} (s) \ ds
\end{equation*}
where $R$ is the positively oriented rectangular contour with vertices $c+i, c+iT, 1-c+iT, 1-c+i$ with $c=1 + \frac{1}{\log T}$. The non-trivial zeros of $\zeta (s)$ up to a height $T$ are contained within $R$ and so by Cauchy's Theorem the integral represents the summation
\[
S= \sum_{0 < \gamma \leq T} \zeta ^{(n)} (\rho)
\]
in question. We then show that the contribution to the integral from the bottom, top and right-hand side of the contour is contained within the error term of the theorem, so the main contribution comes from the integral along the left-hand side of the integral.

In Section~\ref{sect:I}, we evaluate this part of the contour integral. Through this, in Lemma~\ref{Lemma 11} we link the integral to a summation
\[
(-1)^n \sum_{m r \leq \frac{T}{2 \pi}} \Lambda(r) \log^n r
\]
which is then the main object in question.

In Section~\ref{sect:EvaluateS}, Lemma~\ref{lem:SumOverLambdaEqualsLR} we  use Perron's formula to evaluate this sum as a complex integral up a vertical line just to the right of the critical strip.

In Section~\ref{sect: ErrorTerm}, we will evaluate that complex integral and show that 
\[
\frac{1}{2 \pi i } \int_{c- iV}^{c + iV}\left( \frac{\zeta'}{\zeta}(s) \right)^{(n)} \zeta(s) \frac{Y^s}{s} \ ds = \Res_{s=1} \left( \frac{\zeta'}{\zeta}(s) \right)^{(n)} \zeta(s) \frac{Y^s}{s} + E_n(Y)
\]
where $Y=\frac{T}{2\pi}$ and $V$ is given explicitly to optimise the error term $E_n(Y)$, which we can describe both unconditionally and under the assumption of RH. We will also show, via a different argument, a slight improvement on the error term for $n=1$. 

Finally, in Section~\ref{sect:8}, we will evaluate the residue at $s=1$ to find the asymptotic expansion described in Theorem~\ref{General SC} by considering the Laurent expansions around the pole $s=1$ of the terms in the above integrand.

Combining all of these steps together gives the result.

We finish the paper by giving some examples in Section~\ref{sect:9}. Firstly, we can recover the results in Theorem~\ref{Shanks' Conjecture} by specialising to the case $n=1$, and in particular recover Shanks' Conjecture. We then give the first new case where $n=2$, as well as providing some good numerical evidence for the strength of the result.

\section{Preliminary Lemmas}\label{sect:prelim}

We start by stating the strong form of the convexity bound for the zeta function and the $n$\textsuperscript{th} derivatives of the zeta function. The case $n=0$ can be found in Ivi\'c~\cite[Ch. I, Sect. 1.5]{Ivic}, while the cases for all other $n$ can be derived from this using Cauchy's theorem on derivatives of analytic functions on $\zeta (s)$ in a small disc of radius $\frac{1}{\log t}$ centred at $s=\sigma + it$.
\begin{lemma} \label{Convexity}
For $t \geq t_0 > 0$ uniformly in $\sigma$, 
\begin{equation*} \zeta^{(n)} (\sigma + it) \ll
\begin{cases}
t^{\frac{1}{2} - \sigma} \log^{n+1} t &\text{ if $\sigma \leq 0$}\\
t^{\frac{1}{2}(1- \sigma)} \log^{n+1} t &\text{ if $ 0 \leq \sigma \leq 1$}\\
\log^{n+1} t &\text{ if $\sigma \geq 1$}\\
\end{cases}
\end{equation*}
\end{lemma}

The following result follows from Gonek~\cite[Sect. 2, p. 126]{Gonek}. 
\begin{lemma}\label{lem:bound on log deriv zeta}
If $T$ is such that $|T-\gamma|\gg \frac{1}{\log T}$ for any ordinate $\gamma$, uniformly for $-1\leq \sigma \leq 2$ we have
\begin{equation}\label{eq:bound on log deriv zeta}
\left( \frac{\zeta '}{\zeta}(s) \right)^{(n)} \ll \log^{n+2} T
\end{equation}
\end{lemma}

The functional equation for $\zeta(s)$ can be given by
\begin{equation}
\zeta (1-s) = \chi (1-s) \zeta (s) \label{eq:FuncEqu2}
\end{equation}
where
\[
\chi (s) = 2^s \pi^{s-1} \sin \left( \frac{\pi s}{2} \right) \Gamma (1-s)
\]
and $\Gamma(s)$ is the Gamma function. Taking the logarithmic derivative gives
\begin{equation}
\frac{\zeta '}{\zeta} (1-s) = \frac{ \chi '}{\chi} (1-s) - \frac{\zeta '}{\zeta} (s). \label{eq:LogDeriv}
\end{equation}

To obtain the form of the functional equation for $\zeta ^{(n)} (1-s)$ that we will need, we will use the following result from Gonek~\cite{Gonek}, which is his Lemma 6.

\begin{lemma} \label{Gonek Lemma 6}
For $\sigma$ fixed, $n \geq 0$ and $t \geq 1$ we have
\[
\chi^{(n)} (1-s) = (-1)^n \chi (1-s) \log^n \left( \frac{t}{2 \pi} \right) + O (t^{\sigma -\frac{3}{2}} \log^{n-1} t).
\]
\end{lemma}

\begin{lemma}\label{functional zeta(1-s)}
For $\sigma \geq 1$ and $t\geq 1$
\[
\zeta ^{(n)} (1-s) = (-1)^n \chi (1-s) \sum_{k=0}^n \binom{n}{k} \log^{n-k} \left(\frac{t}{2 \pi} \right) \zeta^{(k)} (s) + O(t^{\sigma - \frac{3}{2}} \log^{n} t).
\]
\end{lemma}

\begin{proof}
By the functional equation \eqref{eq:FuncEqu2}, and using the Leibniz product rule 
\begin{equation*}
(-1)^n \zeta^{(n)} (1-s)  = \sum_{k=0}^n \binom{n}{k} (-1)^{n-k} \chi^{(n-k)} (1-s) \zeta^{(k)} (s) 
\end{equation*}
Using Lemma~\ref{Gonek Lemma 6} (and Lemma~\ref{Convexity} to bound $\zeta^{(k)}(s)$ for $\sigma\geq 1$), this equals
\begin{multline*}
\sum_{k=0}^n \binom{n}{k} \chi (1-s)  \log^{n-k} \left( \frac{t}{2 \pi} \right) \zeta^{(k)} (s) + O\left(t^{\sigma - \frac{3}{2}}  \sum_{k=0}^n \log^{n-k-1} t \log^{k+1} t \right) \\
= \sum_{k=0}^n \binom{n}{k} \chi (1-s) \log^{n-k} \left( \frac{t}{2 \pi} \right) \zeta^{(k)} (s) + O(t^{\sigma - \frac{3}{2}} \log^{n} t).
\end{multline*}
Dividing through by the factor of $(-1)^n$ gives the result.
\end{proof}

Finally, another well-known result on the logarithmic derivative of $\chi(s)$ (which can be found in~\cite{Gonek}, for example) allows us to write \eqref{eq:LogDeriv} as
\begin{equation}
\frac{\zeta '}{\zeta} (1-s) = - \log \left(\frac{t}{2 \pi} \right) - \frac{\zeta '}{\zeta} (s) + O\left( \frac{1}{|t|} \right) \label{eq:LogDeriv2}
\end{equation}
valid for any fixed $\sigma$ and $|t|>1$, which we will also need later.

\section{Beginning the proof}\label{sect:beginProof}

Throughout we assume $T$ is sufficiently large and satisfies $|T-\gamma| \gg \frac{1}{\log T}$, where $\gamma$ is the imaginary part of any zero $\rho$. (This constraint simplifies the arguments, but has no effect on the resulting expressions.) Consider the integral
\begin{equation}
S= \frac{1}{2 \pi i} \int_R \frac{\zeta '}{\zeta}(s) \zeta ^{(n)} (s) \ ds \label{eq:S}
\end{equation}
where the contour $R$ is the positively oriented rectangular contour with vertices $c+i, c+iT, 1-c+iT, 1-c+i$ with $c=1 + \frac{1}{\log T}$.

By Cauchy's Residue Theorem,
\[
S= \sum_{0 < \gamma \leq T} \zeta ^{(n)} (\rho).
\]

We need to evaluate $S$ in another way to determine the behaviour of $\sum_{0 < \gamma \leq T} \zeta ^{(n)} (\rho)$. To do this, we begin by splitting the integral $S$ along each part of the contour, so
\begin{align*}
S &= \frac{1}{2 \pi i} \left( \int_{c+i}^{c+iT} + \int_{c+iT}^{1-c+iT} + \int_{1-c+iT}^{1-c+i} + \int_{1-c+i}^{c+i} \right) \frac{\zeta '}{\zeta}(s) \zeta ^{(n)} (s) \ ds \\
&= S^R + S^T + S^L + S^B,
\end{align*}
say. We will first bound $S^B, S^T, S^R$ trivially within error terms as follows. Our main aim will then be to evaluate $S^L$ which we will do in the next section.

\begin{lemma}
The integral along the bottom of the contour is $S^B = O(1)$.
\end{lemma}

\begin{proof}
This follows as the integral is of finite length with a bounded integrand.
\end{proof}

\begin{lemma} \label{main S^T 2}
The integral along the top of the contour is $S^T = O(T^{\frac{1}{2}} \log^{n+2} T)$.
\end{lemma}

\begin{proof}
By the convexity result from Lemma~\ref{Convexity} and using Lemma~\ref{lem:bound on log deriv zeta} to bound the logarithmic derivative, we may write
\begin{align*}
    S^T & \ll \int_{1-c}^{c} \left| \frac{\zeta '}{\zeta} (\sigma + iT) \zeta ^{(n)} (\sigma + iT) \right| \ d\sigma \\
    & \ll \log^2 T \left\{ \int_{1-c}^{0} + \int_{0}^{1} + \int_{1}^{c} \right\} \left| \zeta ^{(n)} (\sigma + iT) \right| \ d\sigma \\
    & \ll \log^2 T \left\{ \int_{1-c}^{0} T^{\frac{1}{2} - \sigma} \log^{n+1} T \ d\sigma + \int_{0}^{1} T^{\frac{1}{2} (1-\sigma)} \log^{n+1} T \ d\sigma + \int_{1}^{c} \log^{n+1} T \ d\sigma \right\} \\
    & \ll T^{\frac{1}{2}} \log^{n+2} T
\end{align*}
since $c=1+ \frac{1}{\log T}$.
\end{proof}

\begin{lemma}
The integral along the RHS of the contour is $S^R = O(\log^{n+3} T)$.
\end{lemma}

\begin{proof}
Along the RHS of the contour we need to evaluate
\[
S^R = \frac{1}{2 \pi i} \int_{c+i}^{c+iT} \frac{\zeta '}{\zeta}(s) \zeta ^{(n)} (s) \ ds.
\]
Writing $\frac{\zeta '}{\zeta}(s)$ and $\zeta ^{(n)} (s)$ in terms of their Dirichlet series, we have
\begin{equation}
\frac{\zeta '}{\zeta}(s) = -\sum_{r=1}^{\infty} \frac{\Lambda(r)}{r^s} \text{ and } \zeta ^{(n)} (s)= (-1)^n \sum _{m=1}^{\infty} \frac{\log ^n m}{m^s} \label{eq:Dirichlet Series}
\end{equation}
where $\Lambda (r)$ is the von Mangoldt function given by
\begin{equation*}
\Lambda(r) = \begin{cases}
       \log p  \quad &\text{if $r = p^k$ for some prime $p$ and some integer $k \geq 1$,} \\
        0 \quad &\text{otherwise.}
    \end{cases}
\end{equation*}
Substituting into $S^R$ we have
\begin{align*}
S^R & = \frac{1}{2 \pi i} \int_{1}^{T} \left( -\sum_{r=1}^{\infty} \frac{\Lambda(r)}{r^{c +it}} \right) \left( (-1)^n \sum _{m=1}^{\infty} \frac{\log ^n m}{m^{c +it}} \right) \ i dt \\
&\ll \sum_{r=1}^{\infty} \sum _{m=1}^{\infty} \frac{\Lambda(r) \log^n m}{r^c m^c} \int_{1}^{T} (rm)^{-it} \ dt \\
&\ll \sum_{r=1}^{\infty} \sum _{m=1}^{\infty} \frac{\Lambda(r) \log^n m}{r^c m^c \log rm} \\
&\ll \frac{\zeta '}{\zeta}(c) \zeta ^{(n)} (c) \\
&\ll \log ^{n+3} T.
\end{align*}
where the final line follows from Lemmas~\ref{Convexity} and \ref{lem:bound on log deriv zeta}.
\end{proof}

Since $S^B, S^T, S^R$ are all within a remainder term of $O(T^{\frac{1}{2}} \log^{n+2} T)$, only the integral over the LHS of the contour will contribute in any meaningful way. Observe that
\[
S^L = \frac{1}{2 \pi i} \int_{1-c+iT}^{1-c+i} \frac{\zeta '}{\zeta}(s) \zeta ^{(n)} (s) \ ds = -\frac{1}{2 \pi i} \int_{c-iT}^{c-i} \frac{\zeta '}{\zeta}(1-s) \zeta ^{(n)} (1-s) \ ds = - \overline{I},
\]
where
\begin{equation}
I = \frac{1}{2 \pi i} \int_{c+i}^{c+iT} \frac{\zeta '}{\zeta}(1-s) \zeta ^{(n)} (1-s) \ ds. \label{eq:I}
\end{equation}

\section{Evaluating $I$}\label{sect:I}

Our main aim then for this section is to evaluate $I$. To do this we want to write $I$ in such a way that we can apply Gonek's Lemma 5 from \cite{Gonek} which we will state below when we use it later in this section.

\begin{lemma}
We can write the integral $I$ given in \eqref{eq:I} as
\begin{align*}
I &= \frac{(-1)^{n+1}}{2 \pi} \int_{1}^{T} \chi (1 - c - it) \left[ \sum_{k=0}^n \binom{n}{k} \log^{n-k+1} \left( \frac{t}{2 \pi} \right) \sum_{m=1}^{\infty} \frac{(-1)^k \log^k m}{m^{c + it}} \right. \\
&\qquad \quad +\left. \sum_{k=0}^n \binom{n}{k} \log^{n-k} \left( \frac{t}{2 \pi} \right) \sum_{m=1}^{\infty} \frac{(-1)^{k+1} \log^k m}{m^{c + it}} \sum_{r=1}^{\infty} \frac{\Lambda (r)}{r^{c+it}} \right] \ dt +O(T^{\frac{1}{2}} \log^{n+2} T).
\end{align*}
\end{lemma}

\begin{proof}
Substituting the results from Lemma~\ref{functional zeta(1-s)} and \eqref{eq:LogDeriv2} into $I$ gives
\begin{align*}
I = \frac{1}{2 \pi i} \int_{c+i}^{c+iT} - &\left( \log \left( \frac{t}{2 \pi} \right) + \frac{\zeta'}{\zeta}(s) + O(t^{-1})  \right) \\
&\left( (-1)^n \chi (1-s) \sum_{k=0}^n \binom{n}{k} \log^{n-k} \left(\frac{t}{2 \pi} \right) \zeta^{(k)} (s) + O(t^{c-\frac{3}{2}} \log^{n} t) \right) \ ds.
\end{align*}
Writing $s=c+it$, expanding the bracket in the integral and simplifying the error term gives
\begin{align}
I &= \frac{(-1)^{n+1}}{2 \pi} \int_{1}^{T} \chi (1 - c - it) \left[ \sum_{k=0}^n \binom{n}{k} \log^{n-k+1} \left( \frac{t}{2 \pi} \right) \zeta^{(k)}(c +it) \right. \notag\\
&+\left. \sum_{k=0}^n \binom{n}{k} \log^{n-k} \left( \frac{t}{2 \pi} \right) \zeta^{(k)}(c+it) \frac{\zeta '}{\zeta} (c +it) \right] \ dt + O(T^{\frac{1}{2}} \log ^{n+2} T). \label{eq:I split}
\end{align}
where we have used $\chi (1-c- it) \ll t^{\frac{1}{2} -c}$, and $\zeta^{(k)}(c+it) \ll \log^{k+1} t$, and $\frac{\zeta '}{\zeta} (c +it)  \ll \log^2 t$. Writing $\zeta^{(k)}(s)$ and $\frac{\zeta '}{\zeta} (s)$ as their Dirichlet series \eqref{eq:Dirichlet Series} and substituting into $I$ gives the result.
\end{proof}

We will now require the use of Lemma 5 from Gonek~\cite[Sect. 4, p. 131]{Gonek}. As it is such an important result for our proof, we state it here, with the phrasing adapted to suit our needs.

\begin{lemma}[Gonek]\label{Gonek 5}
Let $\{b_m\}_{m=1}^{\infty}$ be a sequence of complex numbers such that for any $\varepsilon>0$, $b_m \ll m^\varepsilon$. Let $c>1$ be as before and let $k$ be a non-negative integer. Then for $T$ sufficiently large,
\[
\frac{1}{2 \pi} \int_{1}^{T} \left( \sum_{m=1}^{\infty} b_m m^{-c -it} \right) \chi (1-c-it) \log ^k \left(\frac{t}{2 \pi} \right) \ dt = \sum_{1 \leq m \leq \frac{T}{2 \pi}} b_m \log^k m + O(T^{c-\frac{1}{2}} \log ^k T).
\]
\end{lemma}

We can finally simplify $I$ to get a single sum that we will work on evaluating in the next section.

\begin{lemma}\label{Lemma 11}
The integral $I$ can be written as
\[
I= (-1)^n \sum_{m r \leq \frac{T}{2 \pi}} \Lambda(r) \log^n r + O(T^{\frac{1}{2}}\log^{n+2} T)
\]
where the sum is taken over all integers $m$ and $r$ such that $mr \leq \frac{T}{2\pi}$.
\end{lemma}

\begin{proof}
Before we apply Lemma~\ref{Gonek 5}, we will split the integral $I$ given in \eqref{eq:I split} substituting the Dirichlet series to simplify the argument slightly. To do this, we write
\begin{align*}
I &= \frac{(-1)^{n+1}}{2 \pi } \int_{1}^{T} \chi (1 - c - it) \sum_{k=0}^n \binom{n}{k} \log^{n-k+1} \left( \frac{t}{2 \pi} \right) \sum_{m=1}^{\infty} \frac{(-1)^k \log^k m}{m^{c + it}} \ dt \\
& \quad +\frac{(-1)^{n+1}}{2 \pi } \int_{1}^{T} \chi (1 - c - it) \sum_{k=0}^n \binom{n}{k} \log^{n-k} \left( \frac{t}{2 \pi} \right) \sum_{m=1}^{\infty} \frac{(-1)^{k+1} \log^k m}{m^{c + it}} \sum_{r=1}^{\infty} \frac{\Lambda (r)}{r^{c+it}} \ dt \\
& \quad + O(T^{\frac{1}{2}} \log^{n+2} T) \\
&= J_1 + J_2 +O(T^{\frac{1}{2}} \log^{n+2} T),
\end{align*}
say.

For $J_1$, we have
\begin{align*}
J_1 &= (-1)^{n+1}\sum_{k=0}^n \binom{n}{k} (-1)^k  \frac{1}{2 \pi } \int_{1}^{T} \chi (1 - c - it) \log^{n-k+1} \left( \frac{t}{2 \pi} \right) \sum_{m=1}^{\infty} \frac{ \log^k m}{m^{c + it}} \ dt \\
&= (-1)^{n+1} \sum_{k=0}^n \binom{n}{k} (-1)^k \left( \sum_{m \leq \frac{T}{2 \pi}} \log^k m \log^{n-k+1} m +O(T^{\frac{1}{2}}\log^{n-k+1} T) \right)\\
&= (-1)^{n+1} \sum_{k=0}^n \binom{n}{k} (-1)^k \sum_{m \leq \frac{T}{2 \pi}} \log^{n+1} m + O(T^{\frac{1}{2}}\log^{n+1} T) \\
&=0 + O(T^{\frac{1}{2}}\log^{n+1} T)
\end{align*}
since by the Binomial Theorem,
\[
\sum_{k=0}^n \binom{n}{k} (-1)^k = (1 + (-1))^n = 0.
\]

For $J_2$, we have
\begin{align*}
J_2 &= (-1)^{n+1} \sum_{k=0}^n \binom{n}{k} (-1)^{k+1} \frac{1}{2 \pi } \int_{1}^{T} \chi (1 - c - it)  \log^{n-k} \left( \frac{t}{2 \pi} \right) \sum_{m=1}^{\infty} \frac{ \log^k m}{m^{\sigma + it}} \sum_{r=1}^{\infty} \frac{\Lambda (r)}{r^{c+it}} \ dt \\
&= (-1)^{n+1} \sum_{k=0}^n \binom{n}{k} (-1)^{k+1} \frac{1}{2 \pi } \int_{1}^{T} \chi (1 - c - it)  \log^{n-k} \left( \frac{t}{2 \pi} \right) \sum_{mr=1}^{\infty} \frac{ \Lambda (r) \log^k m }{(mr)^{c + it}} \ dt \\
&= (-1)^{n+1} \sum_{k=0}^n \binom{n}{k} (-1)^{k+1} \left( \sum_{m r \leq \frac{T}{2 \pi}} \Lambda (r) \log^k (m) \log^{n-k} (m r) + O(T^{\frac{1}{2}}\log^{n-k} T)\right)\\
&= (-1)^{n} \sum_{m r \leq \frac{T}{2 \pi}} \Lambda (r) \sum_{k=0}^n \binom{n}{k} (-\log m)^k \log^{n-k} (m r)  + O(T^{\frac{1}{2}}\log^{n} T) \\
&= (-1)^{n} \sum_{m r \leq \frac{T}{2 \pi}} \Lambda (r) [\log (m r) -\log m)]^n + O(T^{\frac{1}{2}}\log^{n} T)\\
&= (-1)^{n} \sum_{m r \leq \frac{T}{2 \pi}} \Lambda (r) \log ^n r + O(T^{\frac{1}{2}}\log^{n} T).
\end{align*}
where the last three sums are over all integers $m$ and $r$ such that $mr \leq \frac{T}{2\pi}$. Our results follow from combining $J_1$ and $J_2$.
\end{proof}

Therefore
\begin{align*}
S^L &= - \overline{I} \\
&= (-1)^{n+1} \sum_{m r \leq \frac{T}{2 \pi}} \Lambda (r) \log ^n r + O(T^{\frac{1}{2}}\log^{n+2} T).
\end{align*}
Since we have shown that the terms from $S^B, S^T, S^R$ are harmless within the error term of $O(T^{\frac{1}{2}} \log^{n+2} T)$, we have that the integral given in equation \eqref{eq:S} is equal to the sum
\begin{equation}\label{eq:SasSumOverLambda}
S= (-1)^{n+1} \sum_{m r \leq \frac{T}{2 \pi}} \Lambda (r) \log ^n r + O(T^{\frac{1}{2}} \log^{n+2} T).
\end{equation}

\section{Evaluating the sum $S$}\label{sect:EvaluateS}
As we have discussed above, all that remains to do is to evaluate the sum in \eqref{eq:SasSumOverLambda}. To do this, we first note that by Perron's formula \cite[Ch. 5, Sect. 5.1]{MNT}, we have
\[
(-1)^{n+1} \sum_{m r \leq \frac{T}{2 \pi}} \Lambda (r) \log ^n r = \frac{1}{2 \pi i } \int_{c- i \infty}^{c + i \infty}\left( \frac{\zeta'}{\zeta}(s) \right)^{(n)} \zeta(s) \frac{Y^s}{s} \ ds
\]
where we set $Y=\frac{T}{2\pi}$ and $c=1 + \frac{1}{\log T}$ as before.

Since we want to be able to evaluate the integral on the RHS, we will modify this slightly and instead use a truncated Perron formula.

\begin{lemma}\label{lem:SumOverLambdaEqualsLR}
For $2 \leq V \leq Y$, as $Y \to \infty$,
\[
(-1)^{n+1} \sum_{m r \leq Y} \Lambda (r) \log ^n r = \frac{1}{2 \pi i } \int_{c- iV}^{c + iV} \left( \frac{\zeta'}{\zeta}(s) \right)^{(n)} \zeta(s) \frac{Y^s}{s} \ ds + R(Y,V),
\]
where
\[
R(Y,V) \ll \frac{Y}{V} \log^{n+2} Y.
\]
\end{lemma}
We will use a specific $V$ later in this section. The choice of $V$ will depend on whether we assume RH or not.

\begin{proof}
If we let $a(k)$ denote the coefficients in Dirichlet series for  $\left( \frac{\zeta'}{\zeta}(s) \right)^{(n)} \zeta(s)$, namely 
\[
a(k) = (-1)^{n+1} \sum_{r \mid k} \Lambda(r) \log^n r 
\]
then by the truncated Perron formula \cite[Ch. 5, Sect. 5.1]{MNT} we have that the integral in the Lemma equals 
\[
\sum_{k\leq Y} a(k) + R(Y,V)
\]
with
\begin{align*}
R(Y,V) &\ll \sum_{\frac{Y}{2} < k < 2Y} |a(k)| \min \left( 1, \frac{Y}{V |Y-k|} \right) + \frac{4^c+Y^{c}}{V} \sum_{k=1}^{\infty} \frac{|a(k)|}{k^{c}} \\
& = A+B,
\end{align*}
Writing $k=mr$ we see $\sum_{k\leq Y} a(k)$ equals the sum in the Lemma. To evaluate the error, note that $|a(k)| \leq \log^{n+1} k$ (with equality only if $k$ is prime). Therefore
\begin{align*}
    A &\ll \log^{n+1} Y \sum_{\frac{Y}{2} < k< 2Y} \min \left( 1, \frac{Y}{V |Y-k|} \right) \\
    &\ll \log^{n+1} Y \sum_{\ell \leq Y} \frac{Y}{V} \frac{1}{\ell} \\
    &\ll \frac{Y}{V} \log^{n+2} Y
\end{align*}

For $B$, since $c=1+\frac{1}{\log T}$, the Dirichlet series converges, and so 
\[
B = \frac{4^c+Y^c}{V} \left| \left( \frac{\zeta'}{\zeta}(c) \right)^{(n)} \zeta(c) \right| \ll \frac{Y}{V} \log^{n+2} Y
\]
where we use the fact that $Y \asymp T$.

Combining $A$ and $B$ gives the required bound on $R(Y,V)$.
\end{proof}

\section{The Error Term} \label{sect: ErrorTerm}
We will show that the error term can be described explicitly as follows, depending on whether we assume RH or not. 

\begin{lemma}\label{ErrorLemma}
For $Y=\frac{T}{2\pi}$, as $Y\to\infty$, we have
\[
\frac{1}{2 \pi i } \int_{c- iV}^{c + iV}\left( \frac{\zeta'}{\zeta}(s) \right)^{(n)} \zeta(s) \frac{Y^s}{s} \ ds = \Res_{s=1} \left( \frac{\zeta'}{\zeta}(s) \right)^{(n)} \zeta(s) \frac{Y^s}{s} + E_n(Y)
\]
where $E_n(Y)$ is an error term given by one of the following two cases:
\begin{enumerate}
\item Unconditionally, by setting $V = \mathrm{e}^{C \sqrt{\log Y}}$, we obtain $E_n(Y) = O \left(Y \mathrm{e}^{-C \sqrt{\log Y}} \right)$.
\item Assuming RH, setting $V=2\pi Y$, we obtain $E_n(Y) = O \left( Y^{\frac{1}{2}} \log^{n+\frac{5}{2}} Y \right)$.
\end{enumerate}
where  $C$ a positive constant that is not necessarily the same in each instance.
\end{lemma}

We will calculate the residue in Section \ref{sect:8}. First, we will show how we can obtain the different expressions for $E_n(Y)$ in the following subsections. We will also show in the last subsection that for $n=1$, we are able to obtain a slightly better error term than that given above or in Fujii~\cite{FujiiShanks}.

\subsection{The Unconditional Case}
From Titchmarsh~\cite[Sect. 3.8, p. 54]{Titchmarsh}, we know there is a positive constant $C$ such that for $c' = 1 - \frac{C}{\log V}$, all the zeros of $\zeta(s)$ are $\gg\frac{1}{\log T}$ away from the line running from $c'-i V$ to $c'+ i V$. By Cauchy's Residue Theorem, the integral is
\begin{multline*}
    \frac{1}{2 \pi i } \int_{c- iV}^{c + iV} \left( \frac{\zeta'}{\zeta}(s) \right)^{(n)} \zeta(s) \frac{Y^s}{s} \ ds = \Res_{s=1} \left( \frac{\zeta'}{\zeta}(s) \right)^{(n)} \zeta(s) \frac{Y^s}{s} \\
    + \frac{1}{2 \pi i } \left(\int_{c' + iV}^{c + iV} + \int_{c' - iV}^{c' + iV} - \int_{c' - iV}^{c - iV} \right) \left( \frac{\zeta'}{\zeta}(s) \right)^{(n)} \zeta(s) \frac{Y^s}{s} \ ds.
\end{multline*}

By Lemmas~\ref{Convexity} and \ref{lem:bound on log deriv zeta}, if $V\ll T$, the integral on the horizontal lines can be estimated as
\begin{equation*}
    \frac{1}{2 \pi i } \int_{c' \pm iV}^{c \pm iV} \left( \frac{\zeta'}{\zeta}(s) \right)^{(n)} \zeta(s) \frac{Y^s}{s} \ ds \ll \log^{n+3} V \frac{Y^c}{V}(c-c') \ll \frac{Y}{V} \log^{n+2} V
\end{equation*}
where we use the fact that $c-c' \ll \frac{1}{\log V}$ and $Y^c \ll Y$ since $c=1+\frac{1}{\log(2\pi Y)}$.

For the integral on the vertical line, we have
\[
    \frac{1}{2 \pi i } \int_{c' - iV}^{c' + iV} \left( \frac{\zeta'}{\zeta}(s) \right)^{(n)} \zeta(s) \frac{Y^s}{s} \ ds \ll Y^{c'} \log^{n+3} V \int_{-V}^{V} \frac{1}{1+|t|} \ dt  \ll  Y^{c'} \log^{n+4} V 
\]
Since $c' = 1-\frac{C}{\log V}$, balancing the two errors comes from taking $V = \exp(C\sqrt{\log Y})$ for some positive constant $C$, and so we have
\[
\frac{1}{2 \pi i } \int_{c- iV}^{c + iV}\left( \frac{\zeta'}{\zeta}(s) \right)^{(n)} \zeta(s) \frac{Y^s}{s} \ ds = \Res_{s=1} \left( \frac{\zeta'}{\zeta}(s) \right)^{(n)} \zeta(s) \frac{Y^s}{s} + O \left(Y \mathrm{e}^{-C \sqrt{\log Y}} \right).
\]

\subsection{The Conditional Case}
Throughout this subsection we assume the Riemann Hypothesis, RH. 

One approach would be to choose $c' = \frac{1}{2} + \frac{1}{\log T}$ to guarantee a zero-free region, and set $V=T=2\pi Y$. In that case, an application of Cauchy's Residue Theorem yields
\[
\frac{1}{2 \pi i } \int_{c- iV}^{c + iV}\left( \frac{\zeta'}{\zeta}(s) \right)^{(n)} \zeta(s) \frac{Y^s}{s} \ ds = \Res_{s=1} \left( \frac{\zeta'}{\zeta}(s) \right)^{(n)} \zeta(s) \frac{Y^s}{s} + O \left( T^{\frac{1}{2} + \varepsilon}\right)
\]
for $\varepsilon >0$, where the horizontal pieces of the contour are estimated in a manner similar to that below, and the vertical piece of the contour uses the bound $\zeta(s) \ll t^{\varepsilon}$ for $\sigma \geq 1/2$.

However, we can get a better bound, one that depends explicitly upon $n$ by choosing $c' = 1- c = - \frac{1}{\log T}$ (that is, just to the left of the critical strip) and $V \ll T$. By Cauchy's Residue Theorem,
\begin{multline*}
    \frac{1}{2 \pi i } \int_{c- iV}^{c + iV} \left( \frac{\zeta'}{\zeta}(s) \right)^{(n)} \zeta(s) \frac{Y^s}{s} \ ds = \Res_{s=1} \left( \frac{\zeta'}{\zeta}(s) \right)^{(n)} \zeta(s) \frac{Y^s}{s} \\ + \sum_{|\gamma|<V} \Res_{s=\rho} \left( \frac{\zeta'}{\zeta}(s) \right)^{(n)} \zeta(s) \frac{Y^s}{s} 
    + \frac{1}{2 \pi i } \left(\int_{c' + iV}^{c + iV} + \int_{c' - iV}^{c' + iV} - \int_{c' - iV}^{c - iV} \right) \left( \frac{\zeta'}{\zeta}(s) \right)^{(n)} \zeta(s) \frac{Y^s}{s} \ ds,
\end{multline*}
where the sum runs through $\rho = \beta + i\gamma$, the zeros of $\zeta (s)$, lying inside the contour.

To estimate the integral on the horizontal lines, choose $V$ so that all the zeros of zeta are bounded by $\gg \frac1{\log V}$ away from the horizontal line. Therefore we may use Lemma~\ref{lem:bound on log deriv zeta} to bound the logarithmic derivative by $\log^{n+2} V$ along the line. Using the convexity result from Lemma~\ref{Convexity}, we obtain, unconditionally,
\begin{align*}
     \frac{1}{2 \pi i } \int_{c' \pm iV}^{c \pm iV} \left( \frac{\zeta'}{\zeta}(s) \right)^{(n)} \zeta(s) \frac{Y^s}{s} ds \ll &\log^{n+2} V \int_{-\frac{1}{\log V}}^0 V^{\frac{1}{2}-\sigma} \log V \frac{Y^{\sigma}}{V} \ d\sigma \\
     &+ \log^{n+2} V \int_0^1 V^{\frac{1}{2}(1-\sigma)} \log V \frac{Y^{\sigma}}{V} \ d\sigma \\
     &+ \log^{n+2} V \int_1^{1+\frac{1}{\log V}} \log V \frac{Y^{\sigma}}{V} \ d\sigma \\
     \ll&  \frac{Y}{V} \log^{n+2} V
\end{align*}

For the integral on the vertical line, since $V\leq T$ and $c'=-\frac{1}{\log T}$ we may again use Lemma~\ref{lem:bound on log deriv zeta} as the vertical line is bounded away from any zeros of zeta. Using Lemma~\ref{Convexity} to bound zeta just to the left of the critical strip, we obtain (again unconditionally),
\begin{equation*}
    \frac{1}{2 \pi i } \int_{c' - iV}^{c' + iV} \left( \frac{\zeta'}{\zeta}(s) \right)^{(n)} \zeta(s) \frac{Y^s}{s} \ ds  \ll \int_1^{V} \log^{n+2}t \  t^{\frac12} \log t \ \frac{1}{t} \ dt 
    \ll V^{\frac{1}{2}} \log^{n+3} V
\end{equation*}

We now consider the poles at $s=\rho$ for each $\rho$ with $|\gamma| < V$, where $\rho$ is a zero of $\zeta (s)$.

First we note that for $-1 \leq \sigma \leq 2$ and $0<t_0 \leq t \leq V$, we have
\[
\frac{\zeta '}{\zeta} (\sigma+i t) = \sum_{|\gamma - t| < 1} \frac{1}{s-\rho} + O(\log V).
\]
Being careful with the error term, we may differentiate this $n$ times to give
\begin{equation*}
    \left( \frac{\zeta '}{\zeta} (\sigma + it) \right)^{(n)} = \sum_{|\gamma - t| < 1} \frac{(-1)^n n!}{(s-\rho)^{n+1}} + O(\log V). 
\end{equation*}

For each $\rho$, we need to consider the coefficient of $(s-\rho)^n$ in the expansion of $\zeta(s) \frac{Y^s}{s}$ to find the residue at $s=\rho$. For this, note that by the triple product rule, we may write
\[
\left( \zeta(s) \frac{Y^s}{s} \right)^{(n)} = \sum_{\substack{k_1+k_2+k_3 = n \\ k_1, k_2, k_3 \geq 0}} \binom{n}{k_1, k_2, k_3} \zeta^{(k_1)} (s) (Y^s)^{(k_2)} \left( \frac{1}{s} \right)^{(k_3)}
\]
where $\binom{n}{k_1, k_2, k_3}$ is the multinomial coefficent given by
\[
\binom{n}{k_1, k_2, k_3} = \frac{n!}{k_1! k_2! k_3!}.
\]

Therefore, at each zero $\rho$,
\begin{equation*}
    \left. \left( \zeta(s) \frac{Y^s}{s} \right)^{(n)} \right\vert_{s=\rho}= \sum_{\substack{k_1+k_2+k_3 = n \\ k_1, k_2, k_3 \geq 0}} (-1)^{k_3} k_3!  \binom{n}{k_1, k_2, k_3} \frac{\zeta^{(k_1)} (\rho) Y^{\rho} \log^{k_2} Y} {\rho^{k_3+1}}.
\end{equation*}

As we are just bounding these terms, we do not worry about the coefficient such as the $n!$, $(-1)^n$ and the multinomial coefficients. Assuming RH, so $\rho = \frac12 + i\gamma$ and summing over all zeros with $|\gamma| < V$, we have 
\begin{equation}\label{eq:ResidueSum}
   \sum_{|\gamma|<V}  \Res_{s=\rho} \left( \frac{\zeta'}{\zeta}(s) \right)^{(n)} \zeta(s) \frac{Y^s}{s} \ll Y^{\frac{1}{2}} \sum_{\substack{k_1+k_2+k_3 = n \\ k_1, k_2, k_3 \geq 0}} \log^{k_2} Y  \sum_{\gamma < V}  \frac{\left| \zeta^{(k_1)} (\frac12 +i\gamma) \right|} {\left| \frac12+i\gamma \right| ^{k_3+1}}.
\end{equation}

Now assuming RH, by Gonek \cite{Gonek} we have that for all positive integers $n$,
\[
\sum_{0< \gamma \leq T} \left|\zeta^{(n)} \left(\frac{1}{2} + i\gamma \right) \right|^2 \ll T \log^{2n+2} T.
\]

We will use this result together with the Cauchy-Schwarz inequality and partial summation to bound each term in the our summation. We see that the terms in the right hand side of \eqref{eq:ResidueSum} can be bounded by
\begin{align*}
    \log^{k_2} Y  \sum_{\gamma < V}  \frac{\left| \zeta^{(k_1)} (\frac12 +i\gamma) \right|} {\left| \frac12+i\gamma \right| ^{k_3+1}} &\ll \log ^{k_2} Y \left( \sum_{\gamma \leq V} \frac{|\zeta^{(k_1)} (\frac{1}{2} + i\gamma)|^2}{\gamma^{1+k_3}}
\right)^{\frac{1}{2}} \left(\sum_{\gamma \leq V} \frac{1}{\gamma^{1+k_3}}\right)^{\frac{1}{2}} \\
&\ll  \log ^{k_2} Y \left( \int_{1}^{V} \frac{t \ \log^{2j+2} t}{t^{2+k_3}} \ dt \right)^{\frac{1}{2}} \left( \int_{1}^{V} \frac{\log t}{t^{1+k_3}} \ dt \right)^{\frac{1}{2}} \\
&\ll 
\begin{cases}
\log ^{k_2} Y \left( \log^{2k_1 + 5} V \right)^{\frac{1}{2}} &\text{ if $k_3 = 0$} \\
\log ^{k_2} Y &\text{ if $k_3 \geq 1$.}
\end{cases}
\end{align*}

Clearly the dominant error term is when $k_3=0$ (which forces $k_2=n-k_1$), and so the sum over the residues at the zeros in \eqref{eq:ResidueSum} is bounded by
\begin{align*}
\sum_{|\gamma|<V} \Res_{s=\rho} \left( \frac{\zeta'}{\zeta}(s) \right)^{(n)} \zeta(s) \frac{Y^s}{s} &\ll \sum_{k_1 = 0}^n Y^{\frac{1}{2}} \log ^{n-k_1} Y \log^{k_1 + \frac{5}{2}} V \\
&\ll Y^{\frac{1}{2}} \log ^{n+\frac52} Y .
\end{align*}

Balancing this error term with the error term from the vertical line comes from taking $V=\frac{Y}{\log Y}$, and so we have
\[
\frac{1}{2 \pi i } \int_{c- iV}^{c + iV}\left( \frac{\zeta'}{\zeta}(s) \right)^{(n)} \zeta(s) \frac{Y^s}{s} \ ds = \Res_{s=1} \left( \frac{\zeta'}{\zeta}(s) \right)^{(n)} \zeta(s) \frac{Y^s}{s} + O \left( Y^{\frac{1}{2}} \log^{n+\frac{5}{2}} Y \right).
\]

\subsection{An improvement on the error term in the conditional case for $n=1$}
In this section we assume RH. When $n=1$ we have an error bound of $E_1(T) = O \left( T^{\frac{1}{2}} \log^{\frac{7}{2}} T\right)$ which agrees with Fujii's result in \cite{FujiiShanks}. 

From the previous section, we see that
\begin{equation*}
\sum_{|\gamma|<V}  \Res_{s=\rho} \left( \frac{\zeta'}{\zeta}(s) \right)' \zeta(s) \frac{Y^s}{s} \ll Y^{\frac12} \sum_{0 < \gamma < V} \frac{\left|\zeta' \left( \frac{1}{2} + i\gamma \right) \right|}{|\frac12+i\gamma|}
\end{equation*}

Garaev \cite{Garaev} has shown that under RH, we have
\[
\sum_{0 < \gamma \leq T} \left|\zeta ' \left( \frac{1}{2} + i\gamma \right) \right| \ll T \log^{\frac{9}{4}} T.
\]
\begin{remark}
Garaev's paper assumes additionally that all the zeros are simple. However, a close examination of the proof shows that this particular result does not need that additional assumption.
\end{remark}
By partial summation  
\begin{equation*}
Y^{\frac{1}{2}} \sum_{0 < \gamma < V} \frac{| \zeta ' ( \frac{1}{2} + i\gamma) |}{\gamma} \ll Y^{\frac{1}{2}} \int_{1}^{V} \frac{t \ \log^{\frac{9}{4}} t}{t^2} \ dt \ll Y^{\frac{1}{2}} \log ^{\frac{13}{4}} V
\end{equation*}
so setting $V= Y/\log^{\frac{3}{2}} Y$ to balance this error term with the error coming from the vertical line in the contour yields
\[
E_1(T) = O\left(T^{\frac12} \log^{\frac{13}{4}} T \right)
\]
an improvement of $\log^{\frac{1}{4}} T$ in our error term.

\section{Finding the Leading Asymptotic Terms}\label{sect:8}
We now evaluate
\[
\Res_{s=1} \left( \frac{\zeta'}{\zeta}(s) \right)^{(n)} \zeta(s) \frac{Y^s}{s}
\]
from Lemma~\ref{ErrorLemma}. We expand each of the terms in this residue calculation in their Laurent expansions about $s=1$. Since
\[
\frac{\zeta ' (s)}{\zeta (s)} = - \frac{1}{s-1} +A_0 + A_1 (s-1) +...+ A_k (s-1)^k + ... + A_n (s-1)^n +O((s-1)^{n+1}),
\]
we have

\begin{enumerate}
\item The Laurent expansion for $\left( \frac{\zeta'}{\zeta}(s) \right)^{(n)}$ about $s=1$:
\[
\left( \frac{\zeta'}{\zeta}(s) \right)^{(n)} = \frac{(-1)^{n+1} n!}{(s-1)^{n+1}} + n! A_n + O((s-1))
\]
\item The Laurent expansion for $\zeta(s)$ about $s=1$:
\[
\zeta(s) = \frac{1}{s-1} +C_0 + C_1 (s-1) +...+ C_k (s-1)^k +... + C_n (s-1)^n +O((s-1)^{n+1})
\]
\item The Laurent expansion for $Y^s$ about $s=1$:
\begin{align*}
Y^s =& Y \left( 1 + (s-1) \log Y +... + \frac{(s-1)^k}{k!} \log^k Y +...+  \frac{(s-1)^n}{n!} \log^n Y \right. \\
& \left.+  \frac{(s-1)^{n+1}}{(n+1)!} \log^{n+1} Y + O((s-1)^{n+2}) \right)
\end{align*}
\item The Laurent expansion for $\frac{1}{s}$ about $s=1$:
\[
\frac{1}{s} = 1 - (s-1) + (s-1)^2 +...+ (-1)^k (s-1)^k +... + (-1)^n (s-1)^n + O((s-1)^{n+1})
\]
\end{enumerate}
We now work out the residue at $s=1$ by considering the terms in powers of $\log^k Y$, as $k$ runs from $k=n+1$ down to $k=0$. We consider a combinatorial style argument to make sure we consider all possible terms. The following tables show this for various powers of $\log^k Y$.

To calculate the resulting contribution to the residue for the leading order behaviour, that is, the coefficient of $Y \log^{n+1} Y$, there is only one way to obtain a factor of $\frac{1}{s-1}$ for the residue calculation. We have to have the $\frac{(-1)^{n+1} n!}{(s-1)^{n+1}}$ term from $\left( \frac{\zeta'}{\zeta}(s) \right)^{(n)}$, the $\frac{1}{s-1}$ term from $\zeta (s)$, the $\frac{(s-1)^{n+1}}{(n+1)!}$ term from $Y^s$ and the $1$ term from $\frac{1}{s}$. Combined, the coefficient for $Y \log^{n+1} Y$ is $\frac{(-1)^{n+1}}{n+1}$.

To calculate the sub-leading term there are two distinct possibilities.
\begin{enumerate}
\item The first possibility is to have the $\frac{(-1)^{n+1} n!}{(s-1)^{n+1}}$ term from $\left( \frac{\zeta'}{\zeta}(s) \right)^{(n)}$, the $C_0$ term from $\zeta(s)$, the $\frac{(s-1)^{n}}{n!}$ term from $Y^s$ and the $1$ term from $\frac{1}{s}$. The resulting contribution to the residue for this term is $(-1)^{n+1} [C_0]$.
\item The other possibility is to have the $\frac{(-1)^{n+1} n!}{(s-1)^{n+1}}$ term from $\left( \frac{\zeta'}{\zeta}(s) \right)^{(n)}$, the $\frac{1}{s-1}$ term from $\zeta(s)$, the $\frac{(s-1)^{n}}{n!}$ term from $Y^s$ and the $-(s-1)$ term from $\frac{1}{s}$. The resulting contribution to the residue for this term is $(-1)^{n+1} [-1]$.
\end{enumerate}
Combined, the coefficient for the $Y \log^{n} Y$ term is $(-1)^{n+1} [C_0 -1]$

The other terms in all lower order cases follow in a similar manner --- Table~\ref{table:3} shows all possibilities, with one extra caveat for the lowest order term which has an additional term.
\FloatBarrier

\begin{table}[htp]\small
\centering
\begin{tabular}{||c c c c | c||}
 \hline
$\left( \frac{\zeta'}{\zeta}(s) \right)^{(n)}$ & $\zeta(s)$ & $Y^s$ & $\frac{1}{s}$ & Contribution to the residue \\ [1ex]
 \hline \hline
$\frac{(-1)^{n+1} n!}{(s-1)^{n+1}}$ & $C_k(s-1)^k$ & $\frac{(s-1)^{n-k}Y \log^{n-k} Y}{(n-k)!}$ & $1$ & $ (-1)^{n+1} \frac{n!}{(n-k)!} C_k Y \log^{n-k} Y$ \\ [1ex]
\hline
$\frac{(-1)^{n+1} n!}{(s-1)^{n+1}}$ & $C_{k-1}(s-1)^{k-1}$ & $\frac{(s-1)^{n-k}Y \log^{n-k} Y}{(n-k)!}$ & $-(s-1)$ & $ (-1)^{n} \frac{n!}{(n-k)!} C_{k-1} Y \log^{n-k} Y$ \\ [1ex]
\hline
\vdots & \vdots & \vdots & \vdots & \vdots \\ [1ex]
\hline
$\frac{(-1)^{n+1} n!}{(s-1)^{n+1}}$ & $C_{1}(s-1)$ & $\frac{(s-1)^{n-k}Y \log^{n-k} Y}{(n-k)!}$ & $(-1)^{k-1}(s-1)^{k-1}$ & $ (-1)^{n+k} \frac{n!}{(n-k)!} C_{1} Y \log^{n-k} Y$ \\ [1ex]
\hline
$\frac{(-1)^{n+1} n!}{(s-1)^{n+1}}$ & $C_{0}$ & $\frac{(s-1)^{n-k}Y \log^{n-k} Y}{(n-k)!}$ & $(-1)^{k}(s-1)^k$ & $(-1)^{n+k+1} \frac{n!}{(n-k)!} C_{0} Y \log^{n-k} Y$ \\ [1ex]
\hline
$\frac{(-1)^{n+1} n!}{(s-1)^{n+1}}$ & $\frac{1}{s-1}$  & $\frac{(s-1)^{n-k}Y \log^{n-k} Y}{(n-k)!}$ & $(-1)^{k+1}(s-1)^{k+1}$ & $ (-1)^{n+k} \frac{n!}{(n-k)!} Y \log^{n-k} Y$ \\ [1ex]
\hline
\end{tabular}
\caption{Table to show which terms in the Laurent expansion of each piece of $L$ combine to form a contribution to the $Y \log^{n-k} Y$ term in the residue ($k=0,...,n$).}
\label{table:3}
\end{table}

In the lowest-leading order case, the coefficient of the $Y$ term, there are all the usual terms for the sub-leading behaviour, with one additional term. The additional term comes from the $n! A_n$ term from $\left( \frac{\zeta'}{\zeta}(s) \right)^{(n)}$, the $\frac{1}{s-1}$ term from $\zeta (s)$, the $1$ term from $Y^s$, and the $1$ term from $\frac{1}{s}$. Combined, the extra term for $Y$ has coefficient $n! A_n$.

Combining all of the terms in the above tables, and putting $Y=\frac{T}{2\pi}$ we have shown that
\begin{align*}
\Res_{s=1}& \left( \frac{\zeta'}{\zeta}(s) \right)^{(n)} \zeta(s) \frac{\left(\frac{T}{2\pi}\right)^s}{s} =
(-1)^{n+1} \frac{1}{n+1} \frac{T}{2 \pi} \log^{n+1} \left( \frac{T}{2 \pi} \right)   \\
&+ (-1)^{n+1} \sum_{k=0}^n \binom{n}{k}  (-1)^ {k}  k! \left(-1 + \sum_{j=0}^k (-1)^j C_j \right) \frac{T}{2 \pi} \log^{n-k} \left( \frac{T}{2 \pi} \right)  +  n! A_n\frac{T}{2 \pi}.
\end{align*}

Combining this with the results from Lemma~\ref{ErrorLemma} and using the appropriate choices for $V$ in Lemma~\ref{lem:SumOverLambdaEqualsLR} we see that
\[
(-1)^{n+1} \sum_{m r \leq \frac{T}{2 \pi}} \Lambda (r) \log ^n r = \Res_{s=1} \left( \frac{\zeta'}{\zeta}(s) \right)^{(n)} \zeta(s) \frac{\left(\frac{T}{2\pi}\right)^s}{s} + E_n(T).
\]
Finally, \eqref{eq:SasSumOverLambda} shows that sum is equal to $S= \sum_{0 < \gamma \leq T} \zeta ^{(n)} (\rho)$ (plus an error that is smaller than $E_n(T)$), and the result follows.

\section{Examples of first moments of specific $n$\textsuperscript{th} derivatives of the Riemann zeta function $\zeta (s)$}\label{sect:9}
The first test of a good general theorem is if it can recover any previously known results. We begin this section by doing exactly that in the $n=1$ case where we recover Shanks' Conjecture and Fujii's asymptotics (Theorem~\ref{Shanks' Conjecture}) from Chapter 1. After that we give another example, as well as showing some numerical data showing that our result is in excellent agreement with the true values. We have chosen $n=2$ to show the most simple new case in full detail. The calculations for higher level derivatives become unwieldy to do by hand so we could employ a computer package to manage this for us to obtain any specific $n$\textsuperscript{th} order derivative that we choose.

\subsection{The case $n=1$ (Shanks' Conjecture)}
We can use Theorem~\ref{General SC} to recover Fujii's asymptotic formula Theorem~\ref{Shanks' Conjecture} (with an improved error term) and hence  Shanks' Conjecture from Section 1. To do this, set $n=1$ in the theorem to obtain
\begin{multline*}
\sum_{0 < \gamma \leq T} \zeta'(\rho) =
 \frac{1}{2} \frac{T}{2 \pi} \log^{2} \left( \frac{T}{2 \pi} \right) + \sum_{k=0}^1 \binom{1}{k}  (-1)^ {k}  k!  \left(-1 + \sum_{j=0}^k (-1)^j C_j \right) \frac{T}{2 \pi} \log^{1-k} \left( \frac{T}{2 \pi} \right) \\
+ A_1 \frac{T}{2 \pi} + E_1(T)
\end{multline*}
Note that
\[
A_1 = 2C_1 - \sum_{j=0}^{0} A_j C_{-j} = 2C_1 -A_0C_0 = 2C_1 - C_0^2
\]
since $A_0=C_0$ (mentioned in the preamble before Theorem~\ref{General SC}). Substituting this into the expression above and simplifying gives
\begin{align*}
\sum_{0 < \gamma \leq T} \zeta' (\rho) =
& \frac{T}{4 \pi} \log^{2} \left( \frac{T}{2 \pi} \right) +(-1 +C_0)\frac{T}{2\pi} \log \left(\frac{T}{2\pi} \right)  + (1-C_0 -C_0^2 +3 C_1)\frac{T}{2 \pi}  \\
& + E_1(T)
\end{align*}
in perfect agreement with Theorem~\ref{Shanks' Conjecture} (with the improved error term from Theorem~\ref{General SC} under RH).

\subsection{The case $n=2$}
We set $n=2$ in Theorem~\ref{General SC} to obtain
\begin{multline*}
\sum_{0 < \gamma \leq T} \zeta''(\rho) =
 - \frac{1}{3} \frac{T}{2 \pi} \log^{3} \left( \frac{T}{2 \pi} \right) - \sum_{k=0}^2 \binom{2}{k}  (-1)^ {k}  k!  \left(-1 + \sum_{j=0}^k (-1)^j C_j \right) \frac{T}{2 \pi} \log^{2-k} \left( \frac{T}{2 \pi} \right) \\
 + 2 A_2  \frac{T}{2 \pi}+ E_2(T)
\end{multline*}
Note that
\[
A_2 = 3C_2 - \sum_{j=0}^{1} A_j C_{1-j} = 3C_2 - A_0C_1 - A_1C_0 = 3C_2 - 3C_0 C_1 + C_0^3
\]
Substituting this into the expression above and simplifying gives
\begin{align*}
\sum_{0 < \gamma \leq T} \zeta'' (\rho) =
& - \frac{T}{6 \pi} \log^{3} \left( \frac{T}{2 \pi} \right) -  (-1 +C_0)\frac{T}{2\pi} \log^2 \left(\frac{T}{2\pi} \right) +   (-2+2C_0 -2C_1) \frac{T}{2 \pi}\log \left(\frac{T}{2\pi} \right) \\
&+ (2-2C_0+2C_1+4C_2-6C_0C_1+2C_0^3)\frac{T}{2\pi}  + E_2(T).
\end{align*}
The error term $E_2(T)$ is $O\left(T e^{-C \sqrt{\log T}}\right)$ unconditionally, and assuming RH can be improved to $O(T^{\frac{1}{2}} \log ^{\frac{9}{2}} T)$.

We show some numerical evidence of the result for $n=2$. The 100,000th zero has imaginary part roughly equal to $T=74920.8$, and the true value of $\sum_{0 < \gamma \leq T} \zeta''(\rho)$ is $-2.93961 \times 10^6 + 228.43i$. Clearly the imaginary part is small as we expect, and we can tell by the graph that the result for the real part of the sum is extremely accurate even for this relatively small number of zeros, since the error never exceeds 2,400.
\begin{figure}[htp]
\centering
\includegraphics[scale=0.25]{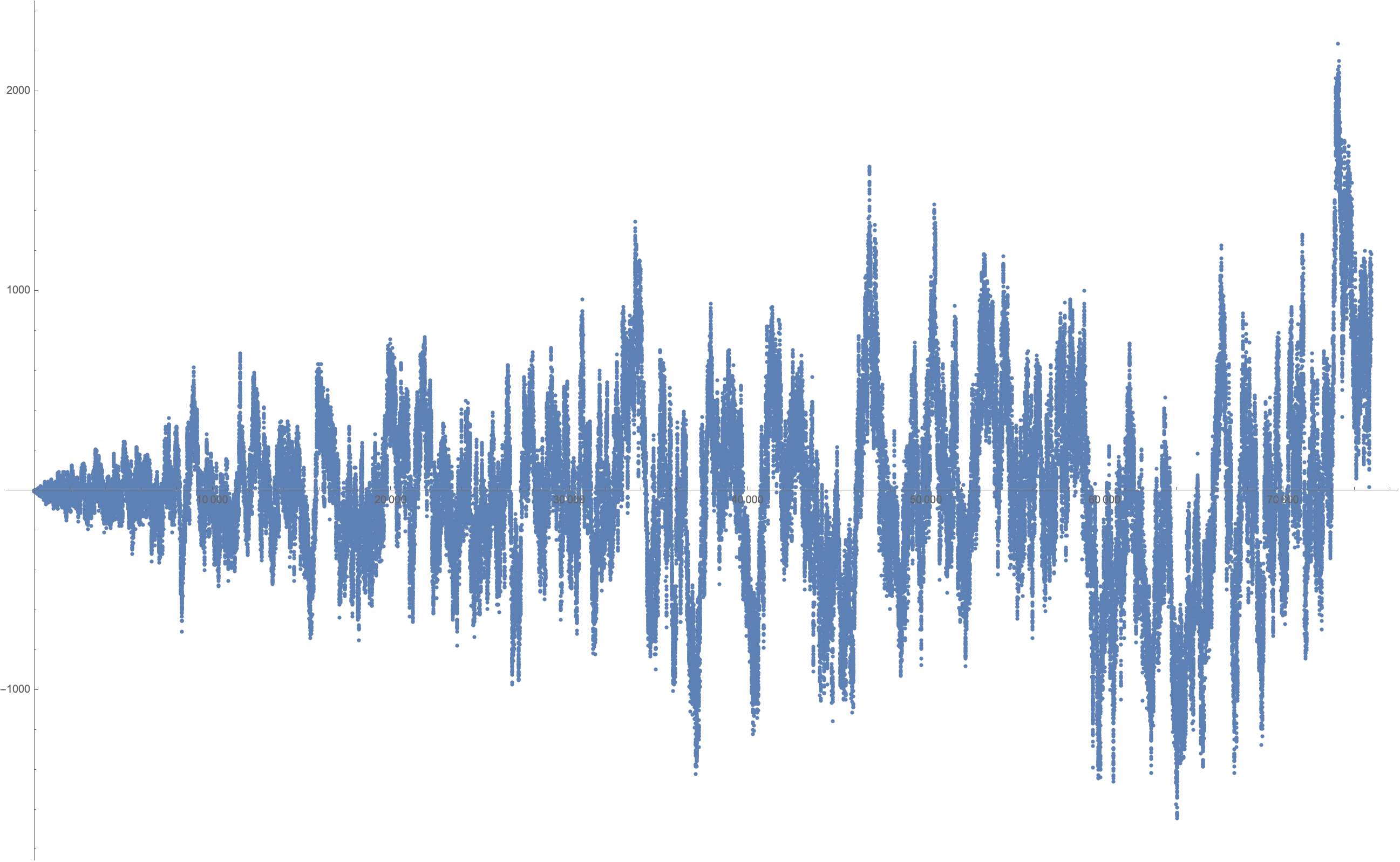}
\caption{Difference in the real part of the actual value of $\sum_{0 < \gamma \leq T} \zeta'' (\rho)$ and the whole asymptotic result of the equation, for $T$ up to the height of the 100,000th zero, showing the real error at each point.}
\end{figure}

\section*{\textbf{Acknowledgements}}
The authors are grateful to the anonymous referee for their comments on the paper. We are particularly thankful for their suggestions with regard to our calculation for $S^T$, which shortened and neatened this argument considerably.

\begin{remark}
This work will form part of the second author's PhD thesis at the University of York \cite{theAPC}.

This research did not receive any specific grant from funding agencies in the public, commercial, or not-for-profit sectors.
\end{remark}

\end{document}